\documentclass{amsart}

\usepackage{amssymb}  
\usepackage{latexsym}
\usepackage{url}
\usepackage[all]{xy}

\DeclareFontEncoding{OT2}{}{} % to enable usage of cyrillic fonts
\newcommand{\textcyr}[1]{%
 {\fontencoding{OT2}\fontfamily{wncyr}\fontseries{m}\fontshape{n}\selectfont #1}}
\newcommand{\Sha}{{\mbox{\textcyr{Sh}}}}

\newcommand{\Z}{{\mathbb Z}}
\newcommand{\Q}{{\mathbb Q}}

\newcommand{\F}{{\mathbb F}}

\newcommand{\PP}{{\mathbb P}}

\newcommand{\kbar}{{\overline{k}}}

\newcommand{\Lbar}{{\overline{L}}}
\newcommand{\Mbar}{{\overline{M}}}

\newcommand{\calH}{{\mathcal H}}

\newcommand{\calX}{{\mathcal X}}
\newcommand{\calY}{{\mathcal Y}}

\newcommand{\fm}{{\mathfrak m}}

\newcommand{\To}{\longrightarrow}

\DeclareMathOperator{\Map}{Map}

\DeclareMathOperator{\End}{End}

\DeclareMathOperator{\Hom}{Hom}

\DeclareMathOperator{\Aut}{Aut}

\DeclareMathOperator{\Br}{Br}

\DeclareMathOperator{\Sel}{Sel}

\DeclareMathOperator{\Div}{Div}
\DeclareMathOperator{\Princ}{Princ}
\DeclareMathOperator{\Pic}{Pic}

\DeclareMathOperator{\PIC}{\bf Pic}
\DeclareMathOperator{\Spec}{Spec}

\DeclareMathOperator{\HH}{H}
\DeclareMathOperator{\diw}{div}

\newcommand{\gal}[1]{\mathfrak{g}_{#1}}
\newcommand{\cdell}{M^\times/\iota(k^\times)\partial(L^\times)}
\newcommand{\res}{\operatorname{res}}

\newcommand{\Cov}{\operatorname{Cov}}
\newcommand{\pr}{\operatorname{pr}}

\newtheorem{Theorem}{Theorem}[section]
\newtheorem{Lemma}[Theorem]{Lemma}
\newtheorem{Proposition}[Theorem]{Proposition}
\newtheorem{Corollary}[Theorem]{Corollary}

\theoremstyle{definition}
\newtheorem{Definition}[Theorem]{Definition}
\newtheorem{Example}[Theorem]{Example}
\newtheorem{Remark}[Theorem]{Remark}

\numberwithin{equation}{section}

\begin{document}

\title[Descent on $\PIC^1(X)$]{Explicit descent in the Picard group of a cyclic cover of the projective line}

\author{Brendan Creutz}
\address{School of Mathematics and Statistics, University of Sydney, NSW 2006, Australia}
\email{brendan.creutz@sydney.edu.au}
\date{10 May 2012}

\begin{abstract}
Given a curve $X$ of the form $y^p = h(x)$ over a number field, one can use descents to obtain explicit bounds on the Mordell--Weil rank of the Jacobian or to prove that the curve has no rational points.  We show how, having performed such a descent, one can easily obtain additional information which may rule out the existence of rational divisors  on $X$ of degree prime to $p$. This can yield sharper bounds on the Mordell--Weil rank by demonstrating the existence of nontrivial elements in the Shafarevich--Tate group. As an example we compute the Mordell--Weil rank of the Jacobian of a genus $4$ curve over $\Q$ by determining that the $3$-primary part of the Shafarevich--Tate group is isomorphic to $\Z/3\times\Z/3$.
\end{abstract}

\maketitle

\section{Introduction}
Let $k$ be a global field and $J/k$ an abelian variety. Any separable isogeny $\varphi:J \to J$ gives rise to a short exact sequence of finite abelian groups, \[ 0 \to J(k)/\varphi(J(k)) \to \Sel^\varphi(J/k) \to \Sha(J/k)[\varphi] \to 0\,, \] relating the finitely generated Mordell--Weil group $J(k)$ and the conjecturally finite Shafarevich--Tate group $\Sha(J/k)$. Computation of the middle term, the {\em $\varphi$-Selmer group of $J$}, is typically referred to as a {\em $\varphi$-descent on $J$}. This produces an explicit upper bound for the Mordell--Weil rank which will only be sharp when $\Sha(J/k)[\varphi]$ is trivial.

While descents on elliptic curves have a history stretching back as far as Fermat, the first examples for abelian varieties of higher dimension appear to have been computed in the 1990s by Gordon and Grant \cite{GordonGrant}, though Cassels had suggested a method using his so called $(x-T)$ map a decade earlier \cite{Casselsg2}. These first examples concerned Jacobians of genus $2$ curves with rational Weierstrass points. Schaefer \cite{Schaefer2d,SchJAC} and Poonen--Schaefer \cite{PoonenSchaefer} later developed a cohomological interpretation of Cassels' $(x-T)$ map which allowed them to generalize the method to Jacobians of all cyclic covers of the projective line. More recently Bruin--Stoll \cite{BruinStoll} and Mourao \cite{Mourao} have used a similar $(x-T)$ map to do a descent on the cyclic cover itself. This computes a finite set of everywhere locally solvable coverings of the curve which may be of use in determining its set of rational points. In particular, when this set is empty there are no rational points on the curve.

We show how, having performed a descent on the Jacobian $J$ of a cyclic cover $X$, one can easily obtain additional information which may rule out the existence of $k$-rational divisors of degree $1$ on $X$. When $X$ is everywhere locally solvable (for instance) the scheme $\PIC^1(X)$, whose $k$-rational points parametrize $k$-rational divisor classes of degree $1$ on $X$, represents an element of $\Sha(J/k)$. So this can be used to show that $\Sha(J/k)$ is nontrivial, and consequently to deduce sharper bounds for the Mordell--Weil rank. We show that this new information can be interpreted as a set parametrizing certain everywhere locally solvable coverings of $\PIC^1(X)$, so one might refer to the method as a {\em descent on $\PIC^1(X)$}. This interpretation allows us to  relate the set in question to the divisibility properties of $\PIC^1(X)$ in $\Sha(J/k)$ (see Theorem \ref{mainthm} and Corollary \ref{Cortomainthm}). Well known properties of the Cassels--Tate pairing then allow us to deduce a better lower bound for the size of $\Sha(J/k)$ (unconditionally). We give several examples. In one we compute the Mordell--Weil rank of the Jacobian of a genus $4$ curve over $\Q$ by determining that the $3$-primary part of the Shafarevich--Tate group is isomorphic to $\Z/3\times\Z/3$. We also present empirical data suggesting better bounds are thus obtained rather frequently for hyperelliptic curves.

While one gets additional information on $k$-rational divisors of degree $1$, this is unlikely to be of much additional use for determining the set of rational points on $X$ when the genus is at least $2$.  When $X(k) \ne \emptyset$, the descent on $\PIC^1(X)$ yields no new information on the Mordell--Weil rank since $\PIC^1(X) \simeq J$. The obstruction to the existence of rational points on $X$ provided by the descent on $\PIC^1(X)$ is weaker than that given by the descent on $X$, and only provides any new information when the descent on $X$ actually gives an obstruction. That being said, descents on $\PIC^1(X)$ could be useful for computing large generators of the Mordell--Weil group or for finding a $k$-rational embedding of $X$ into the Jacobian (see \cite[Section 3.2]{BS2} for some examples with genus $2$ curves), both of which are relevant for tools such as the Mordell--Weil Sieve or Chabauty's Method. Although, such benefits can only be reaped by constructing explicit models for the coverings parametrized by the descent, which is a topic which we will not address here.

\subsection{Notation}
Throughout the paper $p$ will be a prime number and $k$ a field of characteristic different from $p$ containing the $p$-th roots of unity. We use $\kbar$ to denote a separable closure of $k$ and $\gal{k}$ to denote the absolute Galois group of $k$. When $k$ is a global field we denote its completion at a prime $v$ by $k_v$.

If $G$ is a group, a {\em principal homogeneous space for $G$} is a set $H$ on which $G$ acts simply transitively. We make the convention that $\emptyset$ is a principal homogeneous space for any group. Suppose $H$ and $H'$ are principal homogeneous spaces for groups $G$ and $G'$, respectively, and that $i_0:G \to G'$ is a homomorphism of groups. Then a map $i:H \to H'$ is a map is said to be affine (with respect to $i_0$) if $i(g\cdot h) = i_0(g)\cdot i(h)$ for all $h \in H$ and $g \in G$. An {\em affine isomorphism} is an affine bijection with respect to an isomorphism of groups. When $G$ is an abelian group and $n$ is an integer, we use $G[n]$ and $G(n)$ to denote the $n$-torsion subgroup and the subgroup of elements killed by some power of $n$, respectively.

If $L$ is a $k$-algebra we use $\bar{L}$ to denote $L \otimes_k \kbar$. If $V$ is a projective variety over $k$ and $L$ is a commutative $k$-algebra we use $V_{L}$ or $V \otimes_k L$ to denote the extension of scalars, $V \times_{\Spec(k)} \Spec(L)$. The group of $k$-rational divisors on $V$ is denoted $\Div(V)$. The function field of $V$ is denoted $\kappa(V)$. A divisor is called {\em principal} if it is the divisor of a function $f \in \kappa(V)$; the group of all such divisors is denoted $\Princ(V)$. The quotient of $\Div(V)$ by $\Princ(V)$ is denoted $\Pic(V)$. When $V$ is a curve $\Div(V)$ is the free abelian group on the set of closed points of $V$, and there is a well defined notion of degree in $\Div(V)$. For a point $P \in V(\kbar)$ we use $[P]$ to denote the corresponding element in $\Div(V_{\kbar})$. The degree of a principal divisor is $0$, so there is also a well defined notion of degree for classes in $\Pic(V)$. We denote the subset consisting of classes of degree $i$ by $\Pic^i(V)$.

Let $A$ be an abelian variety defined over $k$. A {\em $k$-torsor under $A$} is a variety $T$ over $k$, together with an algebraic group action of $A$ on $T$ defined over $k$ such that the induced map $A\times T \ni (a,t) \mapsto (a+t,t) \in T \times T$ is an isomorphism. This means that geometrically $A$ acts simply transitively on $T$. The $k$-isomorphism classes of $k$-torsors under $A$ are parameterized by the torsion abelian group $\HH^1(k,A)$. The trivial class is represented by $A$ acting on itself by translations, and a $k$-torsor under $A$ is trivial if and only if it possess a $k$-rational point. Thus when $k$ is a global field with completions $k_v$ the Shafarevich--Tate group, $\Sha(A/k) := \ker\left( \HH^1(k,A) \to \bigoplus \HH^1(k_v,A) \right)$, parameterizes isomorphism classes of everywhere locally solvable torsors.

We often refer to a variety as a $k$-torsor under $A$, taking the group action to be implicit. If $T$ is a $k$-torsor under $A$,  then any point $t_0 \in T$ gives rise to an isomorphism $T \simeq A$ defined over $k(t_0)$ sending a point $t \in T$ to the unique $a \in A$ such that $a+t_0 = t$. We say an isomorphism $\psi:T \simeq A$ is {\em compatible with the  torsor structure on $T$} if it is of this type. The action of $A$ on $T$ can be recovered from such an isomorphism by the rule: $a + t = \psi^{-1}(\psi(t)+a)$.

\section{Coverings and divisibility in $\Sha$}
\label{CoveringInterpretation}

\begin{Definition}
\label{DefineCoverings}
Let $\varphi : A' \to A$ be a separable isogeny of abelian varieties. Let $T$ be a $k$-torsor under $A$ and fix a $\kbar$-isomorphism $\psi_T:T\to A$ compatible with the torsor structure. A {\em $\varphi$-covering of $T$} is a $k$-variety $S$ together with a morphism $S \stackrel{\pi}{\to} T$ defined over $k$ such that there exists a $\kbar$-isomorphism $\psi_S:S \to A'$ such that $\varphi\circ\psi_S = \psi_T\circ\pi$. Two $\varphi$-coverings of $T$ are $k$-isomorphic if they are $k$-isomorphic as $T$-schemes. We use $\Cov^\varphi(T/k)$ to denote the set of isomorphism classes of $\varphi$-coverings of $T$. If $k$ is a global field we define the {\em $\varphi$-Selmer set of $T$} to be the subset $\Sel^\varphi(T/k) \subset \Cov^\varphi(T/k)$ consisting of those $\varphi$-coverings which are everywhere locally solvable.
\end{Definition}

We will see below that this definition generalizes the usual definition of the $\varphi$-Selmer group of an abelian variety. The definition does not depend on the choice for $\psi_T$, and the isomorphism $\psi_S$ endows $S$ with the structure of a $k$-torsor under $A'$.

\begin{Lemma}
Let $(S,\pi)$ be a $\varphi$-covering of $T$. Then its group of $\kbar$-automorphisms is isomorphic to $A'[\varphi]$ as a Galois module.
\end{Lemma}

\begin{proof}
Suppose $\psi:S \to S$ is an isomorphism such that $\pi = \pi\circ \psi$ and consider the endomorphism $\tau = \psi_S\circ\psi\circ\psi_S^{-1} - 1 \in \End(A')$. Since $\pi = \varphi \circ \psi_S = \varphi\circ \psi_S \circ \psi$ we have that $\varphi\circ \tau$ is identically $0$. Then $\tau$ is a continuous map from $A'(\kbar)$, which is irreducible, to $A'[\varphi]$, which is discrete. Hence $\tau$ is constant. It follows that $\psi$ is translation by a $\varphi$-torsion point. Conversely it is clear that translation by any $\varphi$-torsion point gives an automorphism of $(S,\pi)$.
\end{proof}

By definition all $\varphi$-coverings of $T$ are twists of one another. So by the twisting principle $\Cov^\varphi(T/k)$ is a principal homogeneous space for the group $\HH^1(k,A'[\varphi])$. In the special case $T = A$ (acting on itself by translations), the morphism $\varphi:A' \to A$ gives $A'$ a canonical structure as a $\varphi$-covering of $A$. This gives a canonical identification of $\HH^1(k,A'[\varphi])$ and $\Cov^\varphi(A/k)$ and consequently endows $\Cov^\varphi(A/k)$ with a group structure in which $\varphi:A'\to A$ represents the identity. Under this identification the isomorphism classes of $\varphi$-coverings of $A$ which possess $k$-rational points correspond to the kernel in the Kummer sequence,
\begin{align}
0 \to A(k)/\varphi(A'(k)) \to \HH^1(k,A'[\varphi]) \to \HH^1(k,A')[\varphi] \to 0\,.
\end{align}
When $k$ is a global field one can deduce from this that $\Sel^\varphi(A/k)$ is identified with the kernel of the natural map $\HH^1(k,A'[\varphi]) \to \bigoplus_{v} \HH^1(k_v,A)$. In particular it is a subgroup and it sits in an exact sequence
\begin{align}
\label{elsKummer}
0 \to A(k)/\varphi(A'(k)) \to \Sel^\varphi(A/k) \to \Sha(A'/k)[\varphi] \to 0\,.
\end{align}

\begin{Remark}
The reader is cautioned that our notation is nonstandard. Our $\Sel^\varphi(A/k)$ would typically be referred to as the $\varphi$-Selmer group of $A'$ (with $A'$ present in the notation). 
\end{Remark}

More generally $\varphi$-Selmer sets are related to divisibility in the Shafarevich--Tate group as follows.

\begin{Proposition}
\label{ShaDiv}
Suppose $\varphi : A ' \to A$ is a separable isogeny of abelian varieties over $k$ and that $T$ is a $k$-torsor under $A$. Then $\Cov^\varphi(T/k) \ne \emptyset$ if and only if $T \in \varphi\HH^1(k,A')$. If $k$ is a global field, then  $\Sel^\varphi(T/k) \ne \emptyset$ if and only if $T \in \varphi\Sha(A'/k)$.
\end{Proposition}

\begin{proof} We will prove the second statement. The first can be proved using the same argument. We may assume $T \in \Sha(A/k)$, otherwise the statement is trivial. Suppose $T$ is killed by $m$ and consider the following commutative and exact diagram: \[ \xymatrix{ \Sel^{m\circ\varphi}(A/k) \ar[d]^{\varphi_*} \ar[r] & \Sha(A'/k)[m\circ\varphi] \ar[r]\ar[d]^\varphi& 0\\ \Sel^{m}(A/k) \ar[r] & \Sha(A/k)[m] \ar[r]& 0 }\] The torsor $T$ admits a lift to an $m$-covering $T \stackrel{\pi}\to A$ in the $m$-Selmer group of $A$. Each choice of lift gives a map $\Sel^\varphi(T/k) \ni (S,\rho) \mapsto (S,\pi\circ\rho) \in \Sel^{(m\circ\varphi)}(A/k)$. The image of this map is exactly the fiber above $(T,\pi)$ under the map denoted $\varphi_*$ in the diagram above. From this one deduces the result from commutativity and the fact that the horizontal maps are surjective.
\end{proof}

We record here the following well known lemma which relates the condition in Proposition \ref{ShaDiv} to the Cassels--Tate pairing.

\begin{Lemma}
\label{CTShaDiv}
Let $\varphi:A'\to A$ be a separable isogeny of abelian varieties over a global field $k$ with dual isogeny $\varphi^\vee:A^\vee \to A'^\vee$. An element of $\Sha(A/k)$ is divisible by $\varphi$ if and only if it pairs trivially with every element of $\Sha(A^\vee/k)[\varphi^\vee]$ under the Cassels--Tate pairing.
\end{Lemma}

\begin{proof}
Compatibility of the Cassels--Tate pairing with isogenies (e.g \cite[I.6.10(a)]{ADT}) shows that it induces a complex \[ \varphi\Sha(A') \to \Sha(A) \to \Hom(\Sha(A^\vee)[\varphi^\vee],\Q/\Z)\,.\] The statement is equivalent to claiming that this is exact. When $\varphi$ is multiplication by an integer this result appears in the paragraph following the proof of \cite[I.6.17]{ADT}. The general statement can be deduced in exactly the same manner.
\end{proof}

\section{Cyclic covers of $\PP^1$}
\label{CyclicCovers}
Let $\pi : X \to \PP^1$ be a cyclic cover of degree $p$ defined over $k$. By the Riemann-Hurwitz formula, $X$ has genus $g = (d-2)(p-1)/2$, where $d$ is the number of branch points of $\pi$. Provided $\PP^1(k)$ has sufficiently many points we can make a change of variables to ensure that $\pi$ is not ramified above $\infty \in \PP^1$. As our present interest lies in infinite fields, there is no harm in assuming this to be the case. The pull back $\mathfrak{m} = \pi^*\infty$ is an effective $k$-rational divisor of degree $p$ on $X$. Let $\Omega \subset X$ denote the set of ramification points of $\pi$. Then for any $\omega \in \Omega$ the divisor $p[\omega]$ is linearly equivalent to $\fm$, and $(2g-2)[\omega]$ is a canonical divisor.

\subsection{The isogeny $\phi$}
Since $k$ contains the $p$-th roots of unity, the group of deck transformations of $\pi$ may be identified with $\mu_p(\kbar)$. The action of $\mu_p(\kbar)$ on $X$ extends linearly to give a Galois equivariant action of the group ring $\Z[\mu_p]$ on $\Div(X_\kbar)$. For any divisor $D$, the element $t = \sum_{\zeta \in \mu_p} \zeta \in \Z[\mu_p]$ sends $D$ to a divisor linearly equivalent to $(\deg D)\fm$. Hence $t$ sends $\Div^0(X_\kbar)$ to $\Princ(X_\kbar)$, so the induced actions of $\Z[\mu_p]$ on $J$ and $\Pic^0(X)$ factor through $\Z[\mu_p]/t$, which is isomorphic to the cyclotomic subring of $k$ generated by $\mu_p$. Fix a generator $\zeta \in \mu_p$ and set $\phi = 1 - \zeta$. Then $\phi:J \to J$ is an isogeny of degree $p^{d-2}$. We note that the ratio of $\phi^{p-1}$ and $p$ is a unit in $\End(J)$.

\subsection{The model $y^p = ch(x)$}
By Kummer theory, $X$ has a (possibly singular) affine model of the form $y^p = ch(x)$, where $c \in k^\times$ and $h(x) \in k[x]$ is a $p$-th-power-free polynomial with leading coefficient $1$. In this model $\pi$ is given by the $x$-coordinate and $\zeta \in \mu_p(\kbar)$ acts via $(x,y) \mapsto (x,\zeta y)$. Our assumption that $\infty$ is not a branch point implies that the branch points are the roots of $h(x)$ and so we may assume $p$ divides the degree of $h(x)$.

\subsection{The torsor $\calX$}
In what follows we consider the (reduced) scheme $\calX = \PIC^1(X)$ classifying linear equivalence classes of divisors of degree $1$ on $X$. This scheme is defined over $k$ and its set of $\kbar$-points is $\calX(\kbar) = \Pic^1(X_{\kbar})$. The obvious injection $\Pic^1(X) \to \Pic^1(X_{\kbar})^{\gal{k}} = \calX(k)$ is not always surjective. The obstruction to a $k$-rational divisor class being represented by a $k$-rational divisor can be interpreted as an element of the Brauer group; one has a well known exact sequence (e.g. \cite[Section 9.1]{BLR})
\begin{align}
\label{thetaX}
0 \to \Pic^1(X) \to \calX(k) \stackrel{\theta_X} \to \Br(k)\,.
\end{align}
The obstruction $\theta_X$ vanishes identically when $\Pic^1(X)$ is nonempty. When $k$ is a global field, the local-global principle for $\Br(k)$ can be used to show that $\Pic^1(X) = \calX(k)$ if $\Pic^1(X_{k_v}) = \emptyset$ for at most one prime. Similarly if $\calX(k_v) = \emptyset$ for at most one prime $v$, then $\Pic^0(X) = \Pic^0(X_{\kbar})^{\gal{k}}$, which is equal to $J(k)$.

There is a $\kbar$-isomorphism $\calX \simeq J\,,$ sending a point $P \in \calX$ corresponding to the divisor class of $D$ to the class of the divisor $D - [\omega_0]$ in $\Pic^0(X_{\kbar}) = J(\kbar)$. This endows $\calX$ with the structure of a $k$-torsor under $J$ (which does not depend on the choice for $\omega_0$). The class of $\calX$ in $\HH^1(k,J)$ is given by the class of the $1$-cocycle sending $\sigma \in \gal{k}$ to the class of $[\omega_0] - [\omega_0^\sigma]$ in $\Pic^0(X_{\kbar}) = J(\kbar)$. As the difference of any two ramification points gives a $\phi$-torsion point on $J$, we see that the class of $\calX$ in $\HH^1(k,J)$ is killed by $\phi$. In particular, this class has order $p$ if and only if $\calX(k) = \emptyset$. This is the case if and only if every $k$-rational divisor class on $X$ has degree divisible by $p$.
\section{The algebraic Selmer set}
\label{AlgebraicSelmerSet1}

\subsection{The $(x-T,y)$ map}
\label{x-ty}
Let $H(x,z)$ be the binary form of degree $n = \deg\left(h(x)\right)$ such that $H(x,1) = h(x)$. Then $X$ is birational to the curve $y^p = cH(x,z)$ in the weighted projective plane $\PP^2(x:y:z)$ with weights $1,n/p,1$. Writing $H(x,z)$ as $H(x,z) = H_1(x,z)^{n_1}\dots H_e(x,z)^{n_e}$ with distinct irreducible factors $H_i(x,z)$, the radical of $H(x,z)$ is $H_{rad}(x,z) = H_1(x,z)\dots H_e(x,z)$. Let $L = \Map_k(\Omega,\kbar) \simeq k[x]/H_{rad}(x,1)$. This is the \'etale $k$-algebra associated to the finite $\gal{k}$-set $\Omega$. It splits as a product $L \simeq K_1\times \dots \times K_e$ of finite extensions of $k$ corresponding to the irreducible factors $h_i(x)$ of $h(x )$. We have a {\em weighted norm map}
\begin{align}
\label{defN}
N:L \simeq K_1\times \dots \times K_e \ni (\alpha_1,\dots,\alpha_e) \mapsto \prod_{i = 1}^eN_{K_i/k}(\alpha_i)^{n_i} \in k\,.
\end{align}
Let
\[\Omega' = \{p[\omega] : \omega \in \Omega\} \cup \left\{ \sum_{\omega\in \Omega} n_\omega[\omega]\right\} \subset \Div(X_\kbar)\,.\] The first set appearing in the union above is isomorphic to $\Omega$ as a $\gal{k}$-set. The divisor $\sum_{\omega\in \Omega} n_\omega[\omega]$ is the zero divisor of the function $y/z^{n/p} \in \kappa(X)^\times$. In particular it is invariant under the action of $\gal{k}$. Thus $\Omega'$ is a disjoint union of $\gal{k}$-sets, and the \'etale $k$-algebra corresponding to $\Omega'$ splits as $\Map_k(\Omega',\kbar) = M \simeq L\times k$. Since the action of $\gal{k}$ on $\Omega'$ is induced from the action on $\Omega$, we have an induced norm map:
\begin{align}
\label{defpartial}
\partial: L = \Map_k(\Omega,\kbar) \ni \alpha \mapsto \left(\omega' = \sum c_\omega[\omega] \mapsto \prod \alpha(\omega)^{c_\omega}\right) \in \Map_k(\Omega',\kbar) = M\,.
\end{align}
Concretely, this is the map \[ L \ni \alpha \mapsto (\alpha^p,N(\alpha)) \in L \times k\,,\] where $N$ is the weighted norm map defined in (\ref{defN}). We can embed $k$ in $M$ via the map $\iota:k \to M \simeq L\times k$ sending $a$ to $(a,a^{n/p})$. The choice is such that $\partial(a) = \iota(a^p)$. 

Let $f \in \Map_k(\Omega',\kappa(X_\kbar)^\times)$ be the map \[
\omega' \mapsto \left\{\begin{array}{cc} (x-x(\omega)z)/z & \text{ if $\omega' = p[\omega]$, }\\ y/z^{n/p} & \text{ if $\omega' = \sum_{\omega\in \Omega}n_\omega[\omega]$. }\end{array}\right.
\] Then $f$ is a Galois equivariant family of functions $f_\omega$ parameterized by $\Omega'$, whose divisors are supported on the union of $\Omega$ and the support of $\fm$. Moreover if $[{\bf w}]\in \Map_k(\Omega,\Div(X_\kbar))$ denotes the map $\left(\omega \mapsto [\omega]\right)$ and we interpret $\iota(\fm)$ as the map
\[
\iota(\fm) = \left(\omega' \mapsto \left\{\begin{array}{cc} \fm & \text{ if $\omega' = p[\omega]$, }\\ \frac{n}{p}\fm & \text{ if $\omega' = \sum_{\omega\in \Omega}n_\omega[\omega]$. }\end{array}\right.\right) \in \Map_k(\Omega',\Div(X_{\kbar}))\,,
\]
then the family of divisors corresponding to $f$ is $\diw(f) = \partial[{\bf w}] - \iota(\fm) \in \Map_k(\Omega',\Div(X_\kbar))\,.$

Following the terminology in \cite{PoonenSchaefer} we will say a divisor is {\em good} if its support is disjoint from $\Omega$ and $\fm$. For any good divisor, $D = \sum_P n_P[P] \in \Div(X_\kbar)$, we may define $f(D) = \prod_Pf(P)^{n_P} \in \Mbar^\times$. Note that if $D \in \Div(X)$, then $f(D) \in M^\times$. Every $k$-rational divisor is linearly equivalent to a good $k$-rational divisor. Using this and applying Weil reciprocity one can prove the following proposition. For details we refer the reader to \cite[Proposition 3.1]{CreutzThesis}, \cite[Section 5]{PoonenSchaefer} or \cite[Section 4]{Unfake}.

\begin{Proposition}
The function $f$ induces a unique homomorphism
\[ f : \Pic(X) \to \cdell \] with the property that the image of the class of any good divisor $D \in \Div(X)$ is given by $f(D)$ as defined above.
\end{Proposition}

\begin{Remark}
The $(x-T,y)$ map of Stoll and van Luijk defined in \cite{Unfake} differs from ours slightly. The second factor of their map is defined using the function $\gamma y/z^{n/p}$ where $\gamma$ is some $p$-th root of $c$. Hence their map and ours agree in degree $0$ only. The projection $\pr_1 : M \simeq L\times k \to L$ induces a map $\cdell \to L^\times/k^\times L^{\times p}$. Composing this with either $f$ or the $(x-T,y)$ map defined in \cite{Unfake} one recovers the $(x-T)$ map defined in \cite{PoonenSchaefer}. The main advantage of our definition over the others is that it defines a homomorphism on all of $\Pic(X)$ and not just the degree divisible by $p$ part. The map used in \cite{BruinStoll,Mourao} to do a descent on $X$ is the restriction of $\pr_1\circ f$ to $X(k) \subset \Pic^1(X)$.
\end{Remark}

Recall that $c \in k^\times$ is the leading coefficient of the polynomial defining $X$. For $r \in \Z$ define
\[ \calH^r_k = \frac{\{ (\alpha,s) \in L^\times\times k^\times \,:\, c^r\cdot N(\alpha) = s^p \}}{(\gamma\alpha^p,\gamma^{n/p} N(\alpha)) \in L^\times\times k^\times \,: \alpha \in L^\times, \gamma \in k^\times \}} \subset \cdell\]

\begin{Lemma}
\label{DescribeHr}
For $r \in \Z$, \[ \calH^r_k = \left(f(D)\partial(\Lbar^\times)\right)^{\gal{k}}/\iota(k^\times)\partial (L^\times)\,.\]
where $D \in \Pic^r(X_\kbar)$ is any divisor class of degree $r$. In particular, $\calH^0_k = (\partial(\Lbar^\times))^{\gal{k}}/\iota(k^\times)\partial(L^\times)$.
\end{Lemma}

\begin{proof}
First we claim that 
\[\partial(\Lbar^\times) = \{ (\alpha,s) \in \Lbar^\times\times \kbar^\times \,:\, N(\alpha) = s^p \}\,.\] By definition $\partial(\Lbar^\times) = \{ (\alpha^p,N(\alpha)) \,:\, \alpha \in \Lbar^\times \}$, so clearly \[\partial(\Lbar^\times) \subset \{ (\alpha,s) \in \Lbar^\times\times \kbar^\times \,:\, N(\alpha) = s^p \}\,.\] For the other inclusion, suppose $(a,s) \in \Lbar^\times\times\kbar^\times$ is such that $N(\alpha)=s^p$. Then for any $p$-th root $\beta \in \Lbar^\times$ of $\alpha$ we have $N(\beta)^p = s^p$. Hence $N(\beta) = \nu s$ for some $\nu \in \mu_p(\kbar)$. Since $h(x)$ is $p$-th power free, the weighted norm map $N:\mu_p(\Lbar) \to \mu_p(\kbar)$ is surjective. Hence there must exist $\nu' \in \mu_p(\Lbar)$ such that $\nu'\beta \in \Lbar^\times$ satisfies $\partial\beta = (\nu'\beta)^p,N(\nu'\beta)) = (\alpha,s)$. This establishes the claim.

For $i = 1,2$, let $\pr_i$ denote the projection of $M \simeq L\times k$ onto the $i$-th factor. For any point $P = (x_0,y_0) \in X$ we have \[cN(\pr_1\circ f(P)) = c\prod_{\omega \in \Omega}(x_0 - x(\omega))^{n_\omega}=ch(x_0) = y_0^p = \pr_2(f(P))^p\,,\] where $n_\omega$ denotes the multiplicity of $\omega$ as a root of $h(x)$. So for any good divisor $D$ of degree $r$ we have $c^rN(\pr_1\circ f(D)) = \pr_2(f(D))^p\,,$ and, in light of the claim above, we have
\[ f(D)\partial(\Lbar^\times) = \{ (\alpha,s) \in \Lbar^\times\times \kbar^\times \,:\, c^r\cdot N(\alpha) = s^p \}\,.\] In particular, the coset $f(D)\partial(\Lbar^\times)$ depends only on the degree of $D$. The same is then true of its Galois invariant subset. The Lemma now follows easily.
\end{proof}

\begin{Corollary}
\label{H1emptyimplies}
If $\calH^1_k = \emptyset$, then $\Pic^1(X) = \emptyset$.
\end{Corollary}

\begin{proof}
The image of $f:\Pic^r(X) \to \cdell$ is contained in $\calH^r_k$.
\end{proof}

\subsection{The algebraic Selmer set}
\label{AlgebraicSelmerSet}
Over a global field one can combine the information from the various local versions of the map $f$ to obtain a finite subset of $\calH^r_k$ which contains the image of $\Pic^r(X)$.

\begin{Definition}
For a global field $k$ with completions $k_v$, define {\em algebraic $\phi$-Selmer sets}:
\begin{align*}
\Sel_{alg}^\phi(J/k) &= \left\{ \delta \in \calH^0_k\,:\,\text{ for all primes $v$, } \res_v(\delta) \in f(\Pic^0(X_{k_v})) \right\}\,,\\
\Sel_{alg}^\phi(X/k) &= \left\{ \delta \in \calH^1_k\,:\,\text{ for all primes $v$, } \res_v(\delta) \in f(X(k_v)) \right\}\,,\\
\Sel_{alg}^\phi(\calX/k) &= \left\{ \delta \in \calH^1_k\,:\,\text{ for all primes $v$, } \res_v(\delta) \in f(\Pic^1(X_{k_v})) \right\}\,.
\end{align*}
\end{Definition}

Recall that the projection $\pr_1 : M \simeq L\times k \to L$ induces a map $\pr_1: \cdell \to L^\times/k^\times L^{\times p}$.  The {\em fake $\phi$-Selmer group} considered in \cite{PoonenSchaefer} is equal to $\pr_1\left(\Sel^\phi_{alg}(J/k)\right)$. The {\em unfaked $\phi$-Selmer group} considered in \cite{Unfake} is equal to $\Sel^\phi_{alg}(J/k)$. Their results show that if $X$ has divisors of degree $1$ everywhere locally, then $\Sel^\phi_{alg}(J/k)$ can be identified with the $\phi$-Selmer group of $J$. In particular, $\Sel_{alg}^\phi(J/k)$ is finite. Provided $\Sel_{alg}^\phi(\calX/k)$ is nonempty it is a coset of $\Sel_{alg}^\phi(J/k)$ inside $\cdell$. This implies that $\Sel_{alg}^\phi(\calX/k)$ is also finite. If in addition, $\delta \in \Sel_{alg}^\phi(X/k) \ne \emptyset$, then $\Sel_{alg}^\phi(\calX/k) = \delta\cdot\Sel_{alg}^\phi(J/k)$. In particular, $\Sel_{alg}^\phi(X/k) \subset \Sel_{alg}^\phi(\calX/k)$. The set $\pr_1\left(\Sel^\phi_{alg}(X/k)\right)$ is equal to the {\em fake $\phi$-Selmer set} considered in \cite{BruinStoll, Mourao} where it is shown to be a quotient of the $\phi$-Selmer set of $X$ (see Definition \ref{Xcov}). As we shall see (Corollary \ref{Xselmer}) $\Sel_{alg}^\phi(X/k)$ is in one to one correspondence with $\phi$-Selmer set of $X$. 

One motivation for considering this set is that it can explain the failure of the Hasse principle for $X$. Similarly, one can easily deduce the implication: \[ \left(\Sel_{alg}^\phi(\calX/k) = \emptyset\right)\,\, \Longrightarrow\,\, \left(\Pic^1(X) = \emptyset\right)\,.\] When $X$ has points everywhere locally we can say even more.

\begin{Theorem}
\label{mainthm}
Suppose $k$ is a global field and $X$ is everywhere locally solvable. Then $\Sel_{alg}^\phi(\calX/k)$ is nonempty if and only if the torsor $\calX$ is divisible by $\phi$ in $\Sha(J/k)$. 
\end{Theorem}

In light of Proposition \ref{ShaDiv}, to prove the theorem it will suffice to show that when $X$ is everywhere locally solvable $\Sel^\phi(\calX/k)$ and $\Sel_{alg}^\phi(\calX/k)$ are in one to one correspondence. This will be accomplished with Proposition \ref{MainCor} below.

\begin{Corollary}
\label{Cortomainthm}
Suppose $k$ is a global field and $X$ is everywhere locally solvable. If $\Sel_{alg}^\phi(\calX/k)$ is empty, then $\dim_{\F_p}\Sha(J/k)[\phi] \ge 2$. If, in addition, $\dim_{\F_p}\Sha(J/k)[\phi] \le 2$, then $\Sha(J/k)(p) \simeq \Z/p\times\Z/p$.
\end{Corollary}

\begin{proof}
Under the assumptions the theorem implies that $\calX$ represents a nontrivial class in the finite abelian group $G = \frac{\Sha(J/k)[\phi]}{\phi\Sha(J/k)[\phi^2]}$. Under the canonical identification of $J$ with its dual, $\phi$ is self dual (up to a unit). It then follows from Lemma \ref{CTShaDiv} and \cite[Corollary 12]{PoonenStoll} that the Cassels--Tate pairing induces a nondegenerate, alternating pairing on $G$. Hence the order of $G$ is a positive even power of $p$. This establishes the first statement. For the second, note that the assumptions imply that $\phi\Sha(J/k)[\phi^2] = 0$, and use that $\phi^{p-1} = p$ up to a unit.
\end{proof}

\begin{Remark}
To show $G$ has square order it is enough to assume that $p$ is odd or that $X$ has a $k_v$-rational divisor of degree $1$ for each prime $v$. We use the assumption that $X$ is everywhere locally solvable to ensure that $\calX$ represents a nontrivial element of $G$. Indeed this assumption is used in our proof of Theorem \ref{mainthm} when we apply Lemma \ref{pointsaregood} in the proof of Proposition \ref{MainCor}. While it may be possible to relax this hypothesis, some assumption on the existence of $k_v$-rational divisors of degree $1$ is required. For the curve $X:y^2 = 3x^6+3$, one can show that $\Pic^1(X_{\Q_2}) = \emptyset$, while $\calX(\Q) \ne \emptyset$. So the algebraic Selmer set is empty, but $\calX \in 2\Sha(J/\Q)$.
\end{Remark}

\begin{Remark}
It is not generally true that $\Sha(J/k)[\phi]$ has square order. Well known examples with $p = 2$ are given in \cite{PoonenSchaefer} and are necessarily explained by the fact that $X$ fails to have a $k_v$-rational divisor of degree $1$ at an odd number of primes. An example with $p = 3$ where $X$ has a rational point is given in \cite{FisherCES}.
\end{Remark}

\subsection{Computing the algebraic Selmer set}
\label{COMPU} Before carrying on with the proof of Theorem \ref{mainthm} we briefly discuss how $\Sel_{alg}^\phi(\calX/k)$ can be computed in practice. For an extension $K/k$ set $\mathfrak{L}(K) = (L \otimes_k K)^\times/K^\times (L \otimes_k K)^{\times p}$, and use $\res_K$ to denote the canonical map $\mathfrak{L}(k) \to \mathfrak{L}(K)$. The weighted norm $N:L \to k$ induces a map $N:\mathfrak{L}(k) \to k^\times/k^{\times p}$. If $k$ is a local field, an element of $\mathfrak{L}(k)$ is said to be {\em unramified} if its image under $\res_{k^{u}}$ is trivial, where $k^{u}$ denotes the maximal unramified extension of $k$. If $k$ is a global field, an element $\delta \in \mathfrak{L}(k)$ is said to be {\em unramified at a prime $v$} of $k$ if $\res_{k_v}(\delta)$ is unramified.

Now suppose $k$ is a global field and let $S$ denote the set of primes of $k$ consisting of all primes of bad reduction, all nonarchimedean primes dividing $cp$ and all archimedean primes\footnote{Actually, one can get away with using a smaller set of primes. Compare with \cite[Cor. 4.7 and Prop. 5.12]{StollIMP}, \cite[Lem. 4.3]{BruinStoll}, \cite[Lem. 2.6]{Mourao}.}. Let $\mathfrak{L}(k)_S$ denote the subgroup of $\mathfrak{L}(k)$ consisting of elements which are unramified at all primes outside of $S$. This is a finite group which can be computed from the $S$-unit group and class group of each of the constituent fields of $L$ (see \cite[12.5--12.8]{PoonenSchaefer}). For an element $a \in k^\times$, let $\mathfrak{L}(k)_{S,a}$ denote the subset of $\mathfrak{L}(k)_S$ consisting of elements $\alpha$ such that $aN(\alpha) \in k^{\times p}$.

Computable descriptions of $\pr_1\left(\Sel_{alg}^\phi(J/k)\right)$ and $\pr_1\left(\Sel_{alg}^\phi(X/k)\right)$ are given in \cite[Theorem 13.2]{PoonenSchaefer}, \cite[Section 6]{BruinStoll} and \cite[Corollary 3.12]{Mourao}. They are the subsets of $\mathfrak{L}(k)_{S,1}$ and $\mathfrak{L}(k)_{S,c}$ cut out by certain local conditions. The former is the subgroup of elements which restrict into $\pr_1\circ f(\Pic^0(X_{k_v}))$ for all $v \in S$ while the latter is the subset which restricts into $\pr_1\circ f(X(k_v))$ for all primes with norm up to some explicit bound. For explicit descriptions of how to compute these local images see \cite{BruinStoll,Mourao,StollIMP}.

\begin{Proposition}
Suppose that $D_v \in X(k_v)$ for each $v \in S$. Then
\[\pr_1\left(\Sel_{alg}^\phi(\calX/k)\right) =  \left\{ \delta \in \mathfrak{L}(k)_{S,c} \,:\, \res_{k_v}(\delta) \in \pr_1(f(D_v))\cdot \pr_1(f(\Pic^0(X_{k_v})))\text{ for all $v \in S$} \right\}\,.\]
\end{Proposition}

\begin{proof}
This follows from the descriptions of  $\pr_1(\Sel^\phi_{alg}(J/k))$ and $\pr_1(\Sel^\phi_{alg}(X/k))$ above and the fact that $\pr_1\circ f$ is a homomorphism.
\end{proof}

\begin{Remark}
This shows that while doing a $\phi$-descent on $J$ (i.e. computing $\pr_1\left(\Sel_{alg}^{\phi}(J/k)\right))$, one can determine whether $\Sel_{alg}^\phi(\calX/k)$ is empty or not with virtually no extra effort.
\end{Remark}

\section{$\phi$-coverings of $X$}
Our proof of Theorem \ref{mainthm} will involve relating $\Sel_{alg}^\phi(\calX/k)$ and $\Sel^\phi(\calX/k)$. To do this we first relate $\Sel_{alg}^\phi(X/k)$ with a certain set of coverings of $X$ which we now define.

\begin{Definition}
\label{Xcov}
A {\em $\phi$-covering of $X$} is a covering $Y \to X$ which arises as the pullback of some $\phi$-covering $\calY \to \calX$ along the canonical map $X \to \calX$ sending a point $P$ to the class of the divisor $[P]$. We use $\Cov^\phi(X/k)$ to denote the set of $k$-isomorphism classes of $\phi$-coverings of $X$. If $k$ is a global field, the {\em $\phi$-Selmer set of $X$} is defined to be the subset $\Sel^\phi(X/k) \subset \Cov^\phi(X/k)$ consisting of those coverings that are everywhere locally solvable.
\end{Definition}

It follows that any $\phi$-covering of $X$ is an $X$-torsor under $J[\phi]$ and that all $\phi$-coverings of $X$ are twists of one another. Hence $\Cov^\phi(X/k)$ is also a principal homogeneous space for $\HH^1(k,J[\phi])$. The action of twisting is compatible with base change, so the obvious map $\Cov^\phi(\calX/k) \to \Cov^\phi(X/k)$ is an affine isomorphism.

Our next goal is to relate $\calH^1_k$ with a certain subset of $\Cov^\phi(X/k)$ and use this to show that $\Sel_{alg}^\phi(X/k)$ and $\Sel^\phi(X/k)$ are in one to one correspondence. While we work with $\Sel_{alg}^\phi(X/k)$ rather than its image under $\pr_1$, this result was essentially established in \cite{BruinStoll,Mourao}. The only new ingredient here is to clarify the affine structure of these sets. This interpretation is, however, crucial to our proof of Theorem \ref{mainthm}.

We have an exact sequence
\begin{align}
\label{phitorsion}
1\to \mu_p \to J_\fm[\phi] \stackrel{q}\to J[\phi] \to 0\,,
\end{align}
where $J_\fm$ is the generalized Jacobian associated to the modulus $\fm \in \Div(X)$ (see \cite[Section 2]{PoonenSchaefer} or \cite[Chapter 5]{SerreAGCF}). Applying Galois cohomology gives an exact sequence
\begin{align}
\label{defineUps}
\HH^1(k,\mu_p) \to \HH^1(k,J_\fm[\phi]) \to \HH^1(k,J[\phi]) \stackrel{\Upsilon}\to \HH^2(k,\mu_p)\,.
\end{align}
The description of $J_\fm[\phi]$ in \cite[Section 6]{PoonenSchaefer} identifies it with the kernel of $\partial:\bar{L}^\times \to \bar{M}^\times$. This allows us to interpret the cocycle in the following proposition as taking values in $J[\phi]$.

\begin{Proposition}
\label{kerUps1}
There is an isomorphism $\calH_k^0 \simeq \ker\Upsilon$ which sends the class of $\partial(\alpha) \in \partial(\bar{L}^\times)^{\gal{k}}$ to the class of the $1$-cocycle $\gal{k}\ni \sigma \mapsto q(\sigma(\alpha)/\alpha) \in J[\phi]$.
\end{Proposition}

\begin{proof}
This can be found in \cite{Unfake} (see Proposition 3.1 and Remark 4.3).
\end{proof}

%\begin{proof}
%In \cite{PoonenSchaefer} it is shown that $J_\fm[\phi]$ may be identified with the kernel of the restriction of the weighted norm map $N: \Lbar^\times \to \kbar^\times$ to $\mu_p(\Lbar)$. This kernel is evidently the same as that of $\partial:\Lbar^\times \to \Mbar^\times$. Thus we can attach exact columns to (\ref{phitorsion}) to obtain a commutative and exact diagram
%\[ \xymatrix{
%1 \ar[r]& \mu_p \ar[r]\ar@{^{(}->}[d]&J_\fm[\phi]\ar[r]\ar@{^{(}->}[d]&J[\phi]\ar[r]&1 \\
%& \kbar^\times \ar@{>>}[d]^p\ar@{^{(}->}[r] & \Lbar^\times \ar@{>>}[d]^\partial& &\\
%& \kbar^\times \ar@{^{(}->}[r]^\iota& \partial(\Lbar^\times) &&\\}\]
%Taking Galois cohomology of the right column and using the fact (Hilbert's Theorem 90) that the $\HH^1$ of the middle term is trivial, one gets that $\HH^1(k,J_\fm[\phi]) \simeq (\partial(\Lbar^\times))^{\gal{k}}/\partial(L^\times)\,.$ The image of this in $\HH^1(k,J[\phi])$ under the map in (\ref{defineUps}) is isomorphic to the quotient by the image of $\HH^1(k,\mu_p)$. Identifying $\HH^1(k,\mu_p)$ with $k^\times/k^{\times p}$ (again using Hilbert's Theorem 90) we find that the image of $\HH^1(k,J_\fm[\phi])$ in $\HH^1(J[\phi])$ is isomorphic to $(\partial(\Lbar^\times))^{\gal{k}}/\iota(k^\times)\partial(L^\times)$. By exactness in (\ref{defineUps}), this image is the same as the kernel of $\Upsilon$. On the other hand, this is isomorphic to $\calH_k^0$ by Lemma \ref{DescribeHr}.
%\end{proof}

\begin{Definition}
\label{ob1}
Define 
\[\Cov_0^\phi(X/k) = \{ (Y,\pi) \in \Cov^\phi(X/k) \,:\, \pi^*[\omega_0] \text{ is linearly equivalent to a $k$-rational divisor} \}\,.\]
\end{Definition}

The pullbacks of the ramification points are all linearly equivalent, so $\pi^*[\omega_0]$ represents a $k$-rational divisor class. If $k$ is a global field and $Y$ is everywhere locally solvable, then every $k$-rational divisor class contains a $k$-rational divisor. Thus we see that $\Sel^\phi(X/k) \subset \Cov_0^\phi(X/k)$.

\begin{Proposition}
\label{DescentOnX}
The action of $\HH^1(k,J[\phi])$ on $\Cov^\phi(X/k)$ restricts to a simply transitive action of $\ker(\Upsilon) \simeq \calH_k^0$ on $\Cov_0^\phi(X/k)$. The function $f$ induces an affine isomorphism \[ \mathfrak{f}:\Cov_0^\phi(X/k) \to \calH_k^1\] with the property that for any $(Y,\pi) \in \Cov_0^\phi(X/k)$ and any extension $K/k$ such that there is a point $Q \in Y(K)$, \[f(\pi(Q))  = \mathfrak{f}((Y,\pi)) \text{ in $\calH_K^1$}. \]
\end{Proposition}

\begin{Corollary}
\label{Xselmer}
Suppose $k$ is a global field. Then $\mathfrak{f}$ restricts to give a bijection $\mathfrak{f}: \Sel^\phi(X/k) \to \Sel_{alg}^\phi(X/k)\,.$
\end{Corollary}

\begin{proof}[Proof of Proposition \ref{DescentOnX}]
Let $(Y,\pi) \in \Cov_0^\phi(X/k)$. The complete linear system associated to $\pi^*[\omega_0]$ gives an embedding in $\PP^N$ (for some $N$) with the property that for $\omega \in \Omega$, the divisor $\pi^*[\omega]$ is a hyperplane section defined by the vanishing of some linear form $l_\omega$. Recall that $[{\bf w}]$ is the map $\left(\omega \mapsto [\omega]\right) \in \Map_K(\Omega,\Div(X_{\kbar}))$. These linear forms $l_\omega$ may be chosen so as to give a linear form $l$ with coefficients in $L$ defining the $\gal{k}$-equivariant family of divisors $\left(\pi^*[{\bf w}] : \omega \mapsto \pi^*[\omega]\right) \in \Map_K(\Omega,\Div(Y_{\kbar}))$. Since the divisor of $f$ is $\partial[{\bf w}] - \iota(\fm) \in \Map_k(\Omega',\Div(X_\kbar))$, we see that there is some $\Delta \in M^\times$ such that \[ \pi^*f = \Delta\frac{\partial(l)}{\iota(z\circ\pi)} \in \Map_k(\Omega',\kappa(Y_{\kbar})^\times)\,.\] Define $\mathfrak{f}((Y,\pi)) = \Delta$. A different choice of model for $Y$ or a different choice for the linear form $l$ would serve to modify $\Delta$ by an element of $\iota(k^\times)\partial(L^\times)$. So the class of $\Delta$ in $\calH^1_k$ is well defined. For any point $Q \in Y(K)$ not lying above a Weierstrass point or some point at above $\infty$ on $X$, the defining property stated in the proposition is immediate. For the finitely many remaining points the result follows by application of the moving lemma.

Given $(\delta,s) \in L^\times \times k^\times$ representing an element of $\calH^1_k$ one can construct a $\phi$-covering of $X$ as follows. Let $\PP_\Omega$ be the projective space with coordinates parameterized by $\Omega$. Define a curve $Y_{\delta,s} \subset \PP_\Omega \times X$ by declaring that $\left( (u_\omega)_{\omega \in \Omega}, (x:y:z) \right) \in Y_{\delta,s}$ if and only if there exists some $a \in k^\times$ such that 
\begin{align}
\label{eqdef}
\delta(\omega)u_\omega^p = a(x-x(\omega)z) \text{, for all $\omega \in \Omega$, and } s\prod_{\omega}u_\omega^{n_\omega}  = a^dy\,.
\end{align}
Recall that $\delta \in L$ can be interpreted as a map $\delta:\Omega \to \kbar$ and that $n_\omega$ denotes the weight associated to $\omega$ in the weighted norm map $N : L \to k$. Projection onto the second factor gives $Y_{\delta,s}$ the structure of an $X$-torsor under $J[\phi]$. It is easy to see that the isomorphism class of $Y_\delta \to X$ depends only on the class of $(\delta,s)$ in $\calH^1_k$. Suppose $(Y,\pi) \in \Cov_0^\phi(X/k)$ and $\mathfrak{f}(Y,\pi) = (\epsilon,t)$. Then (with notation as above) we can find a projective embedding $Y \to \PP^N$ and linear forms $l_\omega$ which cut out the divisors $\pi^*[\omega]$. The rational map $\PP^N \to \PP_\Omega$ given by $(l_\omega)_{\omega \in \Omega}$ gives an isomorphism (of $X$-schemes) $Y \to Y_{\epsilon,t}$. This shows that $Y_{\delta,s}$ constructed are $\phi$-coverings.  It is evident from the construction that the pull back of any ramification point $\omega \in X$ is the hyperplane section of $Y_{\delta,s}$ cut out by $u_\omega = 0$. So this covering represents an element of $\Cov_0^\phi(X/k)$. Moreover, it is clear that the image of $(Y_\delta,\pi_{\delta,s})$ under $\mathfrak{f}$ is represented by $(\delta,s)$. This shows that $\mathfrak{f}$ is surjective.

Now we show that the map is affine with respect to the action of $\calH_k^0 \simeq (\partial(\Lbar^\times))^{\gal{k}}/\iota(k^\times)\partial(L^\times)$. For this suppose $\alpha \in \Lbar^\times$ with $\partial \alpha = (\alpha^p,N(\alpha)) \in (\partial(\Lbar^\times))^{\gal{k}}$. Multiplication by $\alpha$ induces a $\kbar$-automorphism of $\PP_\Omega$. It is evident from (\ref{eqdef}) that this induces an isomorphism of $X$-schemes $\alpha: (Y_{\alpha^p\delta,N(\alpha)s},\pi_{\alpha^p\delta,N(\alpha)s}) \To (Y_{\delta,s},\pi_{\delta,s})$. The cocycle corresponding to this twist is $\xi: \gal{k}\ni \sigma \mapsto \alpha^\sigma\circ\alpha^{-1} \in \Aut((Y_\delta,\pi_{\delta,s})) \simeq J[\phi]$. Under the isomorphism $(\partial(\Lbar^\times))^{\gal{k}}/k^\times \partial(L^\times) \simeq \ker(\Upsilon) \subset \HH^1(k,J_\fm[\phi])$ in Proposition \ref{kerUps1}, the class of $\partial\alpha$ corresponds to the class of the cocycle $\eta: \gal{k} \ni \sigma \mapsto q(\alpha^\sigma/\alpha) \in J[\phi]$, where $q:J_\fm[\phi] \to J[\phi]$ is the quotient map in (\ref{phitorsion}). It is then clear that $\xi$ and $\eta$ give the same class in $\HH^1(k,J[\phi])$. This proves that $\mathfrak{f}$ is affine.
\end{proof}

\section{A descent map for coverings of $\calX$}
We consider the subset $\Cov_{good}^\phi(\calX/k) \subset \Cov^\phi(\calX/k)$ consisting of $\phi$-coverings of $\calX$ such that the corresponding $\phi$-covering of $X$ lies in $\Cov_0^\phi(X/k)$, and define a map \[\mathfrak{F}:\Cov_{good}^\phi(\calX/k) \to \Cov_0^\phi(X/k) \stackrel{\mathfrak{f}}\To \calH^1_k\,.\] Proposition \ref{DescentOnX} implies that $\mathfrak{F}$ is an affine isomorphism.

\begin{Lemma}
\label{pointsaregood}
Suppose $X(k) \ne \emptyset$.
\begin{enumerate}
\item If $(\calY,\pi) \in \Cov^\phi(\calX/k)$ and $\calY(k) \ne \emptyset$, then $(\calY,\pi) \in \Cov_{good}^\phi(\calX/k)$.
\item If $(\calY,\pi) \in \Cov_{good}^\phi(\calX/k)$ and $Q \in \calY(K)$ for some extension $K/k$, then \[ f(\pi(Q)) = \mathfrak{F}((\calY,\pi)) \text{ in $\calH_K^1$}\,.\]
\end{enumerate}

\end{Lemma}
\begin{proof}
By assumption there is some point $R \in X(k) \ne \emptyset$. Then there exists $(Y,\pi) \in \Cov_0^\phi(X/k)$ and $R' \in Y(k)$ such that $\pi(R') = R$. Let $(\calY,\tilde{\pi}) \in \Cov_{good}^\phi(\calX/k)$ be the corresponding covering and $i_Y:Y \to \calY$ the base change of $i_X:X \to \calX$. Clearly $i_Y(R') \in \calY(k) \ne \emptyset$. 

The set $B$ of isomorphism classes of $\phi$-coverings of $\calX$ which contain a $k$-rational point is a principal homogeneous space for the image of $J(k)$ under the connecting homomorphism in the Kummer sequence (\ref{elsKummer}). This image is contained in $\ker(\Upsilon)$. So $B \subset (\calY,\pi)\cdot \ker(\Upsilon) =  \Cov_{good}^\phi(\calX/k)$. This proves (1).

For (2) consider the map $d:\Pic^1(X) \to \Cov^\phi(\calX/k)$ sending a point $P \in \Pic^1(X) = \calX(k)$ to the unique covering to which $P$ lifts. This map is affine, since $f : \Pic^0(X) \to \calH_k^0 \simeq \ker(\Upsilon) \subset \HH^1(k,J[\phi])$ can be identified with the connecting homomorphism in the Kummer sequence \cite[Theorem 1.1]{Unfake}. Moreover its image lands in $\Cov_{good}^\phi(\calX/k)$ by part (1).

It suffices to prove the statement for $K = k$, which amounts to showing that $f(D) = \mathfrak{F}(d(D))$ for every $D \in \Pic^1(X)$. The point $i_Y(R') \in \calY(k)$ is a lift of $[R] \in \calX(k)$, so $\calY = d([R])$. From the definition of $\mathfrak{F}$ and the defining property of $\mathfrak{f}$ we have
\[ \mathfrak{F}((\calY,\tilde{\pi})) = \mathfrak{f}((Y,\pi)) = f([\pi(R')]) = f([R]) \in \calH_k^1\,.\] Hence $f([R]) = \mathfrak{F}(d([R]))$. 

Now suppose $D \in \Pic^1(X)$. Since $d$ is affine, $d(D)$ is the twist of $d([R])$ by the cocyle $f(D-[R]) \in \calH_k^0 \simeq \ker(\Upsilon)$. Since $\mathfrak{F}$ is affine \[\mathfrak{F}(d(D)) = \mathfrak{F}(d[R])\cdot f(D-[R]) = f(D)\mathfrak{F}(d[R])/f([R])=f(D)\,.\] This completes the proof.
\end{proof}

We have the following analogue of Corollary \ref{Xselmer} which, together with Proposition \ref{ShaDiv}, implies Theorem \ref{mainthm}. 

\begin{Proposition}
\label{MainCor}
Suppose $k$ is a global field and $X$ is everywhere locally solvable. Then $\mathfrak{F}$ restricts to an affine isomorphism $\Sel^\phi(\calX/k) \to \Sel^\phi_{alg}(\calX/k)$.
\end{Proposition}

\begin{proof}
First off, let us show that $\Sel^\phi(\calX/k) \subset \Cov_{good}^\phi(\calX/k)$. Suppose that $(\calY,\pi) \in \Sel^\phi(\calX/k)$ and that $X$ is everywhere locally solvable. Consider the covering $\tilde{\pi}:Y \to X$ obtained by pulling back. We want to show that the pull back to $Y$ of some ramification point on $X$ is linearly equivalent to a $k$-rational divisor. The obstruction to a $k$-rational divisor class being represented by a $k$-rational divisor is an element of the Brauer group of $k$. Since the Brauer group of a global field satisfies the local-global principle it suffices to show that $(Y,\tilde{\pi})$ gives a class in $\Cov_0^\phi(X/k_v)$ for every prime $v$. This follows from Lemma \ref{pointsaregood}(1) since we have assumed both $X$ and $\calY$ are everywhere locally solvable.

Now let us show that $\mathfrak{F}$ maps the $\phi$-Selmer set to the algebraic $\phi$-Selmer set. Let $(\calY,\pi) \in \Sel^\phi(\calX/k)$ and set $\delta = \mathfrak{F}((\calY,\pi))$. For every completion $k_v$ of $k$, $X(k_v) \ne \emptyset$, so we may apply Lemma \ref{pointsaregood}(2) over $k_v$. This shows that $\res_v(\delta) \in f(\Pic^1(X_{k_v}))$ for every $v$. Consequently, $\delta$ lies in the algebraic $\phi$-Selmer set. 

It now suffices to show that the map in the statement is surjective, as it is the restriction of an affine isomorphism. For this let $\delta$ be an element in the algebraic $\phi$-Selmer set. Then $\delta \in \calH_k^1$, so $\delta = \mathfrak{F}((\calY,\pi))$ for some $(\calY,\pi) \in \Cov_{good}^\phi(\calX/k)$. We need to show that $\calY$ is everywhere locally solvable. For each prime $v$ we can find $P_v \in \Pic^1(X_{k_v}) \subset \calX(k_v)$ such that $\res_v(\delta) = f(P_v)$. The point $P_v$ lifts to a $k_v$-point on some $\phi$-covering $(\calY_v,\pi_v)$ defined over $k_v$. Moreover $(\calY_v,\pi_v) \in \Cov_{good}^\phi(\calX/k_v)$ by Lemma \ref{pointsaregood}(1) and $\mathfrak{F}((\calY_v,\pi_v)) = \res_v(\delta)$ by Lemma \ref{pointsaregood}(2). Since $\mathfrak{F}$ is injective we have that $\calY \otimes k_v$ and $\calY_v$ are isomorphic, for each prime $v$. This implies that $\calY$ is everywhere locally solvable as required.
\end{proof}

\section{Examples}

We have implemented the algorithm described in Section \ref{COMPU} in the computer algebra system {\tt Magma} \cite{magma} for degree $p$ cyclic covers of $\PP^1$ defined over the $p$-th cyclotomic field. As a test of the algorithm (and the correctness of the implementation) we performed computations for a large sample of hyperelliptic curves. When at all possible we checked our results for consistency with rank bounds obtained by other means (e.g. different implementations of descent on elliptic curves and Jacobians of hyperelliptic curves, points of small height the Jacobian, information obtained assuming standard conjectures, etc.). Some of the resulting data is presented at the end of this section. In addition to this we offer the following examples.

\begin{Example}
The Jacobians of the hyperelliptic curves over $\Q$,
\begin{align*}
X_1 &: y^2 + (x^3+x + 1)y = x^6+5x^5+12x^4 + 12x^3+6x^2-3x-4\\
X_2 &: y^2 + (x^3+x + 1)y = -2x^6+7x^5-2x^4 -19x^3 +2x^2 +18x +7\,,
\end{align*}
have Mordell--Weil rank $0$ and the $2$-primary parts of their Shafarevich--Tate groups are isomorphic to $\Z/2\times\Z/2$.
\end{Example}

\begin{proof}
Let $J_i$ denote the Jacobian of $X_i$ and $\calX_i$ denote $\PIC^1(X_i)$. The $X_i$ are everywhere locally solvable double covers of $\PP^1$.
Using {\tt Magma} we computed that $\Sel^2(J_i/\Q)$ has $\F_2$-dimension $2$ and that the $2$-Selmer set of $\PIC^1(X_i)$ is empty for $i = 1,2$. The result then follows from Theorem \ref{mainthm} and its corollary. 
\end{proof}

\begin{Remark}
These curves were taken from \cite{Evidence} (where they were labeled $C_{125,B}$ and $C_{133,A}$) where it is shown that the order of the $2$-torsion subgroup of $\Sha$ is equal to the order of $\Sha$ predicted by the Birch and Swinnerton-Dyer conjectural formula for several modular Jacobian surfaces. In particular, they proved that the formula holds for those Jacobians considered if and only if $2\Sha = 0$. For the curves considered one can determine the rank (unconditionally) by analytic means, so a $2$-descent on the Jacobian determines $\Sha[2]$, but it only determines $\Sha(2)$ when $\dim_{\F_2}\Sha[2] \le 1$. Apart from the two curves above, all curves considered in \cite{Evidence} had $\dim_{\F_2}\Sha[2] \le 1$. So from the example above one can now conclude for the curves considered in \cite{Evidence} that the conjectural formula holds if and only if $\Sha$ has no elements of odd order.
\end{Remark}

\begin{Example} Let $X/\Q$ be the genus $4$ cyclic cover of $\PP^1$ with affine equation \[X:y^3= 3(x^6+x^4+4x^3+2x^2+4x+3)\,.\] Then $X$ is everywhere locally solvable, yet has no $\Q$-rational divisors of any degree prime to $3$. Moreover, the Jacobian $J$ of $X$ has Mordell--Weil rank $1$ and the $3$-primary part of its Shafarevich-Tate group is isomorphic to $\Z/3\times\Z/3$.
\end{Example}

\begin{proof}
We first note that $X$ is everywhere locally solvable. In order to apply the results of this paper, we work over the field $k = \Q(\zeta_3)$ obtained by adjoining a primitive cube root of unity $\zeta_3$. To prove the result we do $\phi = 1 - \zeta_3$ descents on $J_k$ and $\PIC^1(X_k)$. Using {\tt Magma} we computed that the $\phi$-Selmer group of $J_k$ has $\F_3$-dimension $3$. From The exact sequence (\ref{elsKummer}) it follows that 
\[ \dim_{\F_3}\frac{J(k)}{\phi J(k)} + \dim_{\F_3}\Sha(J/k)[\phi] = 3\,.\] We then computed $\Sel_{alg}^\phi(\PIC^1(X_k)/k)$ and found it to be empty. Using Corollary \ref{Cortomainthm} this lowers the upper bound for the dimension of $J(k)/\phi J(k)$ to $1$.

The divisor on $\PP^1$ defined by $x^3-x^2+4x+4 = 0$ lifts to a degree $3$ $\Q$-rational divisor $D$ on $X$. One can check that the image of the class of $D - \fm$ under $f:\Pic^0(X_k) \to \calH^0_k$ is nontrivial. So we find that $J(k)/\phi J(k)$ has dimension $1$. This gives an upper bound of $2$ for the dimension of $\Sha(J/k)[\phi]$, so by Corollary \ref{Cortomainthm}, $\Sha(J/k)(3) \simeq \Sha(J/k)[\phi]\simeq\Z/3\times\Z/3$. On the other hand, $\PIC^1(X)$ represents an element of $\Sha(J/\Q)[3]$ which is not divisible by $3$ (since it is not divisible by $3$ over $k$). On the other hand the dimension of $\Sha(J/\Q)[3]$ is even \cite{PoonenStoll}, so it is at least $2$. Now the map $\Sha(J/\Q)(3) \to \Sha(J/k)(3)$ obtained by extension of scalars is injective since $[k:\Q] = 2$ is prime to $3$, so we must have $\Sha(J/\Q)(3) \simeq \Z/3\times \Z/3$.

It remains to compute the rank. Galois acts on the ramification points as the full symmetric group, from which it follows that there is no nontrivial $k$-rational $\phi$-torsion in $J(k)$. By \cite[Cor. 3.7 and Prop. 3.8]{SchJAC} it follows that 
\begin{align*}
\text{rank}(J(k)) &= [k:\Q]\cdot \left(\dim\frac{J(k)}{\phi J(k)} - \dim J(k)[\phi]\right) = 2\,, \text{  and}\\
\text{rank}(J(\Q)) &= \frac{\text{rank}(J(k))}{[k:\Q]} = 1 \,.
\end{align*}
In fact, $D-\fm$ represents a point of infinite order in $J(\Q)$. \end{proof}

\begin{Remark}
With only a $\phi$-descent on $J$, one is only able to conclude that $1 \le \dim_{F_3}J(k)/\phi J(k) \le 3$, giving $1 \le \text{rank}(J(\Q)) \le 3$.
\end{Remark}

\begin{Example}[{\bf Data for hyperelliptic curves}]
For $g \in \{ 2,3,4 \}$ we tested our algorithm on various samples of hyperelliptic curves of genus $g$. For varying values of $N$ we randomly chose $10,000$ separable polynomials $h(x) = \sum_{i = 1}^{2g+2} h_ix^i$ with integers $h_i$ bounded in absolute value by $N$ and $(h_{2g+2},h_{2g+1}) \ne (0,0)$, and considered the genus $g$ curves $X$ defined by $y^2 = h(x)$. For each curve considered we computed $\Sel_{alg}^2(J/\Q)$ and $\Sel_{alg}^2(\calX/\Q)$ (assuming the generalized Riemann hypothesis for reasons of efficiency). If the latter was empty, we noted whether or not this was because $\Pic^1(X_{\Q_p})$ was empty for some prime $p \le \infty$. The resulting data is summarized in the table below. The bold faced entries correspond to curves where our algorithm provided information that would have not otherwise been obtained.

It is also interesting to consider how often the combined information yields a sharp upper bound for the Mordell--Weil rank. This will be the case if (i) $\calX$ is either trivial or not divisible by $2$ in $\Sha(J/\Q)$; (ii) The number of primes where $X$ fails to have divisors of degree $1$ locally is at most one (resp. not even and positive when the genus is even); and (iii) $\Sha(J/\Q)[2]$ contains at most two elements linearly independent from $\calX$. The assumptions (i) and (ii) imply, respectively, that in order for $\calX(\Q)$ to be empty it is necessary and sufficient that $\Sel_{alg}^2(\calX/\Q)$ be empty, while (iii) guarantees that determining if $\calX(\Q)$ is empty is sufficient to deduce a sharp bound. 

With this in mind we used a point search to compute a lower bound for the rank for each curve, both with and without assuming that the divisible subgroup of $\Sha(J/\Q)$ is trivial (the assumption allows us to determine the parity of the rank). When this matched the upper bound it means we computed the rank, and in such cases we counted (parenthetically) the number of curves where the additional information provided by $\Sel^2_{alg}(\calX/\Q)$ was needed. We then computed the proportion of curves for which the rank could be determined with the additional information provided by our algorithm among those for which the rank could not be determined by descent on the Jacobian alone. 

For example, in the sample of genus $2$ curves with  $N = 10$ the method yielded new information for ca. 17\% of the curves, which (assuming $\Sha_{div}=0$) increased our success rate from ca. $76\%$ to ca. $93\%$, handling ca. 73\% of the curves left previously undecided by the descent on the Jacobian.

\begin{table}[thb]
\begin{tabular}{|r||r|r||r|r||r|r|r|} \hline
&\multicolumn{2}{c|}{$\Sel^2(\calX/\Q)=\emptyset$}&\multicolumn{2}{c|}{Rank computed}& \multicolumn{2}{c|}{$\frac{\text{newly determined}}{\text{previously undetermined}}$}\\ \hline

$(g,N)$&$\Pic^1(\Q_p) = \emptyset$ && assuming GRH & and $\Sha_{div} = 0$ & GRH &$\Sha_{div} = 0$  \\
\hline
(2,5) & 1165 & {\bf 981} & 7873 ({\bf 848}) & 9819 ({\bf 977}) & 29\% & 84\%\\
(2,10)&1310 & {\bf 1778} & 5315 ({\bf 1295}) & 9346 ({\bf 1752}) & 22\% & 73\%\\
(2,20)&1367 & {\bf 2420} & 3411 ({\bf 1350}) & 8392 ({\bf 2317}) & 17\% & 59\% \\
(2,50)& 1381 & {\bf 2916} & 2156 ({\bf 1350}) & 6955 ({\bf 2637}) & 15\% & 46\% \\\hline
(3,5)& 873 & {\bf 1228} & 2540 ({\bf 645}) & 7573 ({\bf 1164})  & 8\% & 32\% \\
(3,10)&944 & {\bf 1857} & 1477 ({\bf 786}) & 5840 ({\bf 1619}) & 8\% & 29\%\\\hline
(4,5)& 713 & {\bf 1278} & 1717 ({\bf 484}) & 6031 ({\bf 1127})& 6\% & 22\% \\
(4,10) & 735 & {\bf 1952} & 1296 ({\bf 726}) & 5145 ({\bf 1644}) & 8\% & 25\%\\\hline
\end{tabular}
\end{table}

\end{Example}
\vfill
\pagebreak

\end{document}